\documentclass[12pt,halfline,a4paper]{ouparticle}
\usepackage{amsthm,amsmath,amssymb,mathrsfs}
\usepackage[colorlinks=true,citecolor=black,linkcolor=black,urlcolor=blue]{hyperref}
\usepackage{graphicx}
\usepackage{float}
\usepackage{epstopdf}
\usepackage{epsfig}
\usepackage{extarrows,chngpage,array,float,eqnarray}     

\theoremstyle{plain}
\newtheorem{theorem}{Theorem}[section]
\newtheorem{lemma}[theorem]{Lemma}
\newtheorem{corollary}[theorem]{Corollary}

\newtheorem{observation}[theorem]{Observation}

\theoremstyle{definition}
\newtheorem{definition}[theorem]{Definition}

\newtheorem{proposition}[theorem]{Proposition}
\newtheorem{question}[theorem]{Question}

\theoremstyle{remark}
\newtheorem{remark}[theorem]{Remark}

\allowdisplaybreaks

\begin{document}

\title{A deletion-contraction formula and  monotonicity properties for the polymatroid Tutte polynomial}

\author{Xiaxia Guan$^{a,b,c}$,~~Xian'an Jin$^{b}$,~~Tam\'{a}s K\'{a}lm\'{a}n$^{c}$ \\
\small $^a$Department of Mathematics, Taiyuan University of Technology, P. R. China\\
\small $^b$School of Mathematical Sciences, Xiamen University, P. R. China\\
\small $^c$Department of Mathematics, Institute of Science Tokyo, Japan\\
\small \emph{Email addresses}: guanxiaxia@tyut.edu.cn; xajin@xmu.edu.cn; kalman@math.titech.ac.jp}

\abstract{The Tutte polynomial is a fundamental invariant of  matroids. The polymatroid Tutte polynomial $\mathscr{T}_{P}(x,y)$, introduced by Bernardi,  K\'{a}lm\'{a}n, and Postnikov, is an extension of the classical Tutte polynomial from matroids to polymatroids $P$. In this paper, we first obtain a deletion-contraction formula for $\mathscr{T}_{P}(x,y)$. Then we prove two natural properties of coefficientwise monotonicity, one for containment and one for minors, both for the interior polynomial $x^{n}\mathscr{T}_{P}(x^{-1},1)$ and the exterior polynomial $y^{n}\mathscr{T}_{P}(1,y^{-1})$, where $P$ is a polymatroid over $[n]$. We show by an example that these monotonicity properties do not extend to $\mathscr{T}_{P}(x,y)$. Using deletion-contraction, we obtain  formulas for the coefficients of terms of degree $n-1$ in $\mathscr{T}_{P}(x,y)$. Finally, we characterize hypergraphs $\mathcal{H}=(V,E)$ such that the coefficient of $y^{k}$ in the exterior polynomial of the associated polymatroid $P_{\mathcal{H}}$ attains its maximal value $\binom{|V|+k-2}{k}$ for all $k$ up to some bound.}
\date{\today}

\keywords{Tutte polynomial; Polymatroid; Deletion-contraction formula; Monotonicity}

\maketitle

\section{Introduction}
\label{sec1}

The Tutte polynomial \cite{Tutte} is an important and well-studied topic in graph and matroid theory, having wide applications in statistical physics, knot theory and so on. As a generalization of the one-variable evaluations $T_G(x,1)$ and $T_G(1,y)$ of the Tutte polynomial $T_G(x,y)$ of graphs $G$ to hypergraphs, K\'{a}lm\'{a}n \cite{Kalman1} introduced the interior polynomial $I_{\mathcal{H}}(x)$ and the exterior polynomial $X_{\mathcal{H}}(y)$ for connected hypergraphs $\mathcal{H}$ via internal and external activities of hypertrees. (Hypertrees were first described as `left or right degree vectors' in \cite{Postnikov}.)  Later, K\'{a}lm\'{a}n, Murakami, and Postnikov \cite{Kalman2,Kalman3} established  that certain leading terms of the HOMFLY polynomial \cite{Jones}, which is a generalization of the celebrated Jones polynomial \cite{Jones} in knot theory, of any special alternating link coincide with the common interior polynomial of the pair of hypergraphs derived from the Seifert graph (which is a bipartite graph) of the link.

Integer polymatroids are a generalization of matroids and an abstraction of hypergraphs. Throughout this paper, only the set of bases of integer polymatroids will be considered, and simply called polymatroids. In \cite{Bernardi}, Bernardi, K\'{a}lm\'{a}n, and Postnikov proposed the polymatroid Tutte polynomial $\mathscr{T}_{P}(x,y)$ for polymatroids $P$, which can be reduced to the classical Tutte polynomial $T_{M}(x,y)$ of any matroid $M$.  More precisely,
if $M$ is a matroid of rank $d$ over $[n]$, and  $P=P(M)$ is its corresponding polymatroid, then
\begin{equation}\label{TM} T_{M}(x,y)=\frac{(x+y-xy)^{n}}{x^{n-d}y^{d}}\mathscr{T}_{P}\left(\frac{x}{x+y-xy},\frac{y}{x+y-xy}\right).
\end{equation}
Moreover, for a polymatroid $P$ over $[n]$, the polynomial $\mathscr{T}_{P}(x,y)$ contains both the interior polynomial $I_{P}(x)$ and  the exterior polynomial $X_{P}(y)$ as special cases in the sense that
\begin{equation}\label{IX-T}
I_{P}(x)=x^{n}\mathscr{T}_{P}(x^{-1},1)\ \ \ \text{and}\ \ \ X_{P}(y)=y^{n}\mathscr{T}_{P}(1,y^{-1}),
\end{equation}
 respectively. We remark that $\mathscr{T}_{P}(x,y)$ is equivalent to another polynomial introduced by Cameron and Fink \cite{Cameron}.

One of the most basic properties of the classical Tutte polynomial is its deletion-contraction formula. Let $P$ be a polymatroid on $[n]$ and let $f\colon 2^{[n]}\rightarrow \mathbb{Z}$ be its rank function. For any $t\in [n]$, let $r_{t}=f(\{t\})+f([n]\setminus \{t\})-f([n])$.
In Proposition 4.11 (f) of \cite{Bernardi}, Bernardi et al.\ obtained a deletion-contraction relation of $\mathscr{T}_{P}(x,y)$, for the cases $r_{t}=0$ and $r_{t}=1$, that generalizes the deletion-contraction formula of the classical Tutte polynomial. They also posed the following question.
\begin{question}\label{question} \cite{Bernardi}
Does there exist a deletion-contraction relation in the case $r_{t}>1$?
\end{question}
We answer this question in the affirmative by proving Theorem \ref{deletion-contraction-J}. It is based on the observation that certain natural ``slices'' of polymatroids are again polymatroids. We note that  even if a polymatroid is hypergraphical, not all of its slices will be necessarily so. Let us also mention that in \cite{Chavez}, the authors introduced a polynomial invariant similar (but not equivalent) to $\mathscr{T}_{P}(x,y)$ and established (in a limited sense) a deletion-contraction relation for it.

In 1972, Brylawski \cite{Brylawski} proved that if a matroid $M_{1}$ is a minor of a connected matroid $M_{2}$, then $T_{M_{1}}(x,y)\leq T_{M_{2}}(x,y)$, that is, each coefficient of $T_{M_{1}}(x,y)$ is less than or equal to the corresponding coefficient of $T_{M_{2}}(x,y)$. Unfortunately, we show by an example in Remark \ref{counterexample} that this monotonicity property does not hold for the polymatroid Tutte polynomial, not even if we make the substitution in it suggested by \eqref{TM}.

In \cite{Stanley},  Stanley showed  that if $\Delta$ and $\Delta'$ are lattice polytopes in $\mathbb{R}^{m}$ with $\Delta'\subseteq \Delta$, then $h^{\ast}_{\Delta'}\leq h^{\ast}_{\Delta}$, where $h^{\ast}$ denotes Ehrhart's $h^{\ast}$-polynomial.
In \cite{Kato} Kato observed, based on results of K\'{a}lm\'{a}n and Postnikov \cite{Kalman3}, that the interior polynomial of any connected hypergraph is equal to the $h^{\ast}$-polynomial of the so-called root polytope of the associated bipartite graph.
 The two theorems together imply a natural monotonicity property of the interior polynomial of hypergraphs, that is, for two connected hypergraphs $\mathcal{H}$ and $\mathcal{H}'$, if
 the associated bipartite graph  of  $\mathcal{H}'$  is a subgraph of the associated bipartite graph  of  $\mathcal{H}$, then $I_{\mathcal{H}'}(x)\leq I_{\mathcal{H}}(x)$.

 In this paper we extend this monotonicity property to the interior polynomial and the exterior polynomial of polymatroids, see Theorems \ref{subset-monotonicity} and \ref{minor-monotonicity}. Namely, we show that for two polymatroids $P$ and $P'$, if $P'\subseteq P$ or $P'$ is a minor of $P$, then $I_{P'}(x)\leq I_{P}(x)$ and $X_{P'}(y)\leq X_{P}(y)$. In hypergraphical cases, we obtain that if the associated bipartite graph of $\mathcal{H}'=(V',E')$  is a subgraph of the associated bipartite graph of $\mathcal{H}=(V,E)$ (even if $\mathcal{H}$ and $\mathcal{H}'$ are not necessarily connected), then $I_{\mathcal{H}'}(x)\leq I_{\mathcal{H}}(x)$, if $V'\subseteq V$ and $E'\subseteq E$, then  $X_{\mathcal{H}'}(y)\leq X_{\mathcal{H}}(y)$ (see Corollary \ref{HG-monotonicity}).

We also examine certain individual coefficients of the polymatroid Tutte polynomial. Let $P$ be a polymatroid over $[n]$. From Proposition 4.11 (d) of \cite{Bernardi}, one knows that $[x^{n-k}y^{k}]\mathscr{T}_{P}(x,y)=\binom{n}{k}$ for any $k\in[n]\cup\{0\}$. In other words, the top degree terms of $\mathscr{T}_{P}(x,y)$ form a standard binomial expansion. Let $f$ be the rank function of the polymatroid $P$. In \cite{Guan3}, the authors computed the coefficients $[x^{n-1}]\mathscr{T}_{P}(x,1)=\sum\limits_{i\in [n]}f(\{i\})-f([n])$ and  $[y^{n-1}]\mathscr{T}_{P}(1,y)=\sum\limits_{i\in [n]}f([n]\setminus \{i\})-(n-1)f([n])$. From there it is easy to get formulas for the coefficients $[x^{n-1}y^{0}]\mathscr{T}_{P}(x,y)$
and $[x^{0}y^{n-1}]\mathscr{T}_{P}(x,y)$. As a generalization, using our deletion-contraction relation of $\mathscr{T}_{P}(x,y)$, we  obtain $$[x^{n-k}y^{k-1}]\mathscr{T}_{P}(x,y)=\sum_{\stackrel{S\subseteq [n]}{|S|=k-1}}f(S)+\sum_{\stackrel{S'\subseteq [n]}{|S'|=k}}f(S')-\binom{n}{k-1}f([n])-k\binom{n}{k}$$ for any $k\in[n]$ in Theorem \ref{coe-degree-n-1'}.
Likewise,  formulas for the coefficients $[x^{n-2}y^{0}]\mathscr{T}_{P}(x,y)$ and $[x^{0}y^{n-2}]\mathscr{T}_{P}(x,y)$ can be obtained as well (see Theorem \ref{coe-degree-n-2}).
They imply formulas for the coefficients $[x^{n-2}]\mathscr{T}_{P}(x,1)$
and $[y^{n-2}]\mathscr{T}_{P}(1,y)$ (see Corollary \ref{n-2-coe}).

Connectivity plays a critical role in graphs and hypergraphs. In Theorem \ref{connected-exterior2}, using our deletion-contraction formula and monotonicity property of the exterior polynomial, we prove that for all $k\geq0$ and for any bipartite graph $G=(V\cup E,\mathcal{E})$, the subgraph $G-E'$ is connected for all $E'\subseteq E$ with $|E'|=k$, if and only if $[y^{i}]X_{\mathcal{H}}(y)=\binom{|V|+i-2}{i}$ for any $i\leq k$, where $\mathcal{H}$ is the hypergraph induced by $G$ such that $E$ is the set of hyperedges. We remark that the coefficient $[y^{i}]X_{\mathcal{H}}(y)$ can never be higher than $\binom{|V|+i-2}{i}$. Moreover, this implies that for a connected graph $G=(V,E)$ and an integer $k\geq 0$, $[y^{i}]T_{G}(1,y)=\binom{|E|-i-1}{|V|-2}$ for all $i$ with $|E|-|V|+1-k\leq i\leq |E|-|V|+1$ if and only if $G$ is $(k+1)$-edge connected.

 The paper is organized as follows. In Section 2, we make some necessary preparations. Section 3 is devoted to the deletion-contraction relation of the polymatroid Tutte polynomial. Our monotonicity properties of the interior polynomial and the exterior polynomial for polymatroids are proven in Section 4. In Section 5, we apply the deletion-contraction formula of Section 3 to derive formulas for the coefficients $[x^{n-k}y^{k-1}]\mathscr{T}_{P}(x,y)$ for any $k\in[n]$,
$[x^{n-2}y^{0}]\mathscr{T}_{P}(x,y)$, and $[x^{0}y^{n-2}]\mathscr{T}_{P}(x,y)$
 for polymatroids $P$ and to characterize the connectivity of hypergraphs using the exterior polynomial.
We end with posing some open questions.

\section{Preliminaries}
\label{sec2}

In this section, we will give some definitions and summarize some known results used in the next sections. Throughout the paper, let  $[n]=\{1,2,\ldots,n\}$, $2^{[n]}=\{I|I\subseteq[n]\}$, and let $\mathbf{e}_{1},\mathbf{e}_{2},\ldots,\mathbf{e}_{n}$ denote the canonical basis of $\mathbb{R}^{n}$.
We first recall the definition of polymatroids.
\begin{definition} \label{def polymatroid}
A \emph{polymatroid} $P=P_{f}\subseteq \mathbb{Z}^{n}$ (in other words, over the ground set $[n]$) with rank function $f$ is  given as $$P=\left\{(a_{1},\ldots,a_{n})\in \mathbb{Z}^{n}\bigg|\sum_{i\in I}a_{i}\leq f(I) \ \text{for any}\ I\subseteq [n] \ \text{and} \ \sum_{i\in [n]}a_{i}=f([n])\right\},$$ where   $f\colon 2^{[n]}\rightarrow \mathbb{Z}$ satisfies

\begin{enumerate}
\item[(i)] $f(\emptyset)=0$;

\item[(ii)] $f(I)+f(J)\geq f(I\cup J)+f(I\cap J)$ for any $I,J\subseteq [n]$ (submodularity).
\end{enumerate}
\end{definition}

Polymatroids are non-empty. It is also easy to see that the base polytope of any matroid (more precisely, the set of its vertices)
is a polymatroid. We remark that in other papers on the subject, cf. \cite{Edmonds}, polymatroids are often defined as slightly larger sets, and the set $P$ of Definition \ref{def polymatroid} is referred to as the set of integer bases of a polymatroid.
We, too, will call the elements of $P$ \emph{bases}.

Conversely, if $P$ is a polymatroid on $[n]$, then its rank function $f=f_{P}\colon 2^{[n]}\rightarrow \mathbb{Z}$ can be recovered as
\begin{equation}\label{RF}
f_{P}(I)=\max_{\mathbf{a}\in P}\sum_{i\in I} a_{i}, \ \text{for any subset}\ I\subseteq [n].
\end{equation}
 We refer the readers to \cite{Edmonds} and chapter 44 of \cite{Schrijver} for more details.

A \emph{hypergraph} is a  pair $\mathcal{H}=(V,E)$, where $V$ is a finite set and $E$ is a finite multiset of non-empty subsets of $V$. Elements of $V$ are called \emph{vertices} and  elements of $E$ are called \emph{hyperedges}, respectively, of the hypergraph. For a hypergraph $\mathcal{H}=(V,E)$, its \emph{associated bipartite graph} $\mathrm{Bip} \mathcal{H}$ is defined as follows.
The sets $V$ and $E$ are the partite sets of $\mathrm{Bip} \mathcal{H}$, and an element $v$ of $V$ is connected to an element $e$ of $E$ in $\mathrm{Bip} \mathcal{H}$ if and only if $v\in e$.

For a subset $E'\subseteq E$, let $\mathrm{Bip} \mathcal{H}|_{E'}$ denote the bipartite graph formed by $E'$, all edges of $\mathrm{Bip} \mathcal{H}$ incident with elements of $E'$, and their endpoints in $V$. Define $\mu(E')=0$ for $E'=\emptyset$, and $\mu(E')=|\bigcup E'|-c(E')$ for $E'\neq \emptyset$, where $\bigcup E'=V\cap (\mathrm{Bip} \mathcal{H}|_{E'})$ and $c(E')$ is the number of connected components of $\mathrm{Bip} \mathcal{H}|_{E'}$.
It is known that $\mu$ is submodular; see, e.g., \cite{Kalman1} for a proof.
In that sense, polymatroids are an abstraction of hypergraphs. The elements of the polymatroid induced by $\mu$ (in the sense of Definition \ref{def polymatroid}) will be referred  to as  hypertrees because they are essentially degree distribution of spanning forests of $\mathrm{Bip} \mathcal{H}$, cf. \cite{Kalman1}. A polymatroid is called a \emph{hypergraphical polymatroid}, denoted by $P_{\mathcal{H}}$, if it is the set of all hypertrees of some hypergraph $\mathcal{H}$.

A nonempty finite subset $B\subseteq \mathbb{Z}^{n}$ is a polymatroid on $[n]$ if and only if $B$ satisfies the following two conditions \cite{Herzog}.

\begin{enumerate}
\item[(i)] For any $\mathbf{a}\in B$ and $\mathbf{b}\in B$, we have $\sum\limits_{i\in [n]}a_{i}=\sum\limits_{i\in [n]}b_{i}$.
\item[(ii)] For any $\mathbf{a}\in B$, $\mathbf{b}\in B$, and any $i\in [n]$ with $a_{i}>b_{i}$, there exists some $j\in [n]$ such that $a_{j}<b_{j}$ and $\mathbf{a}+\mathbf{e}_{j}-\mathbf{e}_{i}\in B$, as well as $\mathbf{b}+\mathbf{e}_{i}-\mathbf{e}_{j}\in B$.
\end{enumerate}

Condition (ii) is called the \emph{Exchange Axiom}. This description easily implies that adding the same vector of $\mathbb{Z}^{n}$ to each element  of  a polymatroid yields another polymatroid.

We next recall the (internal and external) activities of a basis of a polymatroid. Note that the following definition relies on the natural order of the set $[n]$.

\begin{definition}
Let $P$ be a polymatroid  over $[n]$.
For a basis $\mathbf{a}\in P$, an index $i\in [n]$ is \emph{internally active} if $\mathbf{a}-\mathbf{e}_{i}+\mathbf{e}_{j}\notin P$ for any $j<i$. Let $\mathrm{Int}(\mathbf{a})=\mathrm{Int}_{P}(\mathbf{a})\subseteq [n]$ denote the set of all
internally active indices with respect to $\mathbf{a}$. Furthermore, let $\iota(\mathbf{a})=|\mathrm{Int}(\mathbf{a})|$ and $\overline{\iota}(\mathbf{a})=n-|\mathrm{Int}(\mathbf{a})|$.

For a basis $\mathbf{a}\in P$, we call $i\in [n]$ \emph{externally active} if $\mathbf{a}+\mathbf{e}_{i}-\mathbf{e}_{j}\notin P$ for any $j<i$. Let $\mathrm{Ext}(\mathbf{a})=\mathrm{Ext}_{P}(\mathbf{a})\subseteq [n]$ denote the set of all externally active indices with respect to $\mathbf{a}$, and let $\epsilon(\mathbf{a})=|\mathrm{Ext}(\mathbf{a})|$ and $\overline{\epsilon}(\mathbf{a})=n-|\mathrm{Ext}(\mathbf{a})|$.
\end{definition}

Let $P$ be a polymatroid  over $[n]$ and $f$ be its rank function. For a basis $\mathbf{a}\in P$, let
$$\mathcal{I}(\mathbf{a}):=\mathcal{I}_{P}(\mathbf{a})=\left\{I\subseteq [n]\bigg| \sum_{i\in I}a_{i}= f(I)\right\}.$$
We call this the set of \emph{tight} sets for $\mathbf{a}$.

We now state a conclusion obtained in Lemma 4.2 in \cite{Bernardi}.

\begin{theorem} \cite{Bernardi}\label{ia}
Let $P$ be a polymatroid  over $[n]$. For any $\mathbf{a}\in P$,
\begin{itemize}
  \item [(i)] an index $i\in [n]$ is internally active with respect to $\mathbf{a}$
if and only if there exists a subset $I\subseteq [n]$ such that $i=\min(I)$ and $[n]\setminus I\in \mathcal{I}(\mathbf{a})$;
  \item [(ii)] an index $i\in [n]$ is externally active with respect to $\mathbf{a}$
if and only if there exists a subset $I'\subseteq [n]$ such that $i=\min(I')$ and $I'\in \mathcal{I}(\mathbf{a})$.
\end{itemize}
\end{theorem}

To close this section, we will recall some basics of the polymatroid Tutte polynomial.

\begin{definition}\cite{Bernardi}\label{pop}
Let $P$ be a polymatroid  over $[n]$. The \emph{polymatroid Tutte polynomial} $\mathscr{T}_{P}(x,y)$ is defined as
$$\mathscr{T}_{P}(x,y):=\sum_{\mathbf{a}\in P}x^{oi(\mathbf{a})}y^{oe(\mathbf{a})}(x+y-1)^{ie(\mathbf{a})},$$
where $oi(\mathbf{a}):=|\mathrm{Int}(\mathbf{a})\setminus \mathrm{Ext}(\mathbf{a})|$, $oe(\mathbf{a}):=|\mathrm{Ext}(\mathbf{a})\setminus \mathrm{Int}(\mathbf{a})|$, and $ie(\mathbf{a}):=|\mathrm{Int}(\mathbf{a})\cap \mathrm{Ext}(\mathbf{a})|$.
\end{definition}

Note that the smallest index $i=1$ must be simultaneously internally and externally active  for any basis of any polymatroid. Thus the polynomial $\mathscr{T}_{P}(x,y)$ is always divisible by $x+y-1$.

For a polymatroid $P$ over $[n]$, K\'{a}lm\'{a}n \cite{Kalman1} defined
\emph{the interior polynomial} $$I_{P}(x):=\sum_{\mathbf{a}\in P}x^{\overline{\iota}(\mathbf{a})}$$ and \emph{the exterior polynomial} $$X_{P}(y):=\sum_{\mathbf{a}\in P}y^{\overline{\epsilon}(\mathbf{a})}.$$ It is easy to verify that if $P$ is a polymatroid over $[n]$, then
 $I_{P}(x)=x^{n}\mathscr{T}_{P}(x^{-1},1)$ and $X_{P}(y)=y^{n}\mathscr{T}_{P}(1,y^{-1})$. Moreover, the coefficients of $I_{P}(x)$ and $X_{P}(y)$ are non-negative integers.

Let $P$ be a polymatroid  over $[n]$. Let the symmetric group $S_{n}$ act on $\mathbb{Z}^{n}$ by permutations of coordinates. I.e., for a permutation
$w\in S_{n}$ and a point $\mathbf{a}= (a_{1},\ldots, a_{n})\in \mathbb{Z}^{n}$, define
$w(\mathbf{a}): = (a_{w(1)},\ldots, a_{w(n)})$, and
$w(P) := \{w(\mathbf{a})| \mathbf{a} \in P\}.$ For any $\mathbf{c}=(c_{1},\ldots,c_{n})\in \mathbb{Z}^{n}$, let
$P+\mathbf{c} : =\{(a_{1}+c_{1},\ldots,a_{n}+c_{n})|(a_{1},\ldots,a_{n})\in P\}$.
The \emph{dual} polymatroid  of  a polymatroid $P$ is 
$-P:=\{(-a_{1},\ldots,-a_{n})|(a_{1},\ldots,a_{n})\in P\}$.

The following invariance properties of the polymatroid Tutte polynomial have been proven in \cite{Bernardi}.
\begin{theorem}\cite{Bernardi}\label{invariants}
Let $P$ be a polymatroid on $[n]$. Then
\begin{itemize}
 \item [(i)] $\mathscr{T}_{P}(x,y)=\mathscr{T}_{P+\mathbf{c}}(x,y)$ for any $\mathbf{c}\in \mathbb{Z}^{n}$ (translation invariance);
  \item [(ii)] $\mathscr{T}_{P}(x,y)=\mathscr{T}_{w(P)}(x,y)$ for any $w\in S_{n}$ ($S_{n}$-invariance);
  \item [(iii)] $\mathscr{T}_{P}(x,y)=\mathscr{T}_{-P}(y,x)$, where $-P$ is the dual polymatroid of $P$ (duality).
\end{itemize}
\end{theorem}

Regarding Theorem \ref{invariants} (ii), we note again that the order $1<2<\ldots<n$ plays an implicit role in Definition \ref{pop}, but it does turn out that $\mathscr{T}_{P}$ depends only on $P$ and not on this order. In particular, $I_{P}$ and $X_{P}$ also depend only on $P$, that is, they also satisfy $S_{n}$-invariance. Moreover, for $I_{P}$ and $X_{P}$, duality takes the form
\begin{equation}\label{duality}
I_{-P}=X_{P} \quad \text{and} \quad I_{P}=X_{-P}.
\end{equation}
These two claims are equivalent as $-(-P)=P$.

\section{A deletion-contraction formula}
\label{sec3}

In this section, we will study a deletion-contraction formula of the polymatroid Tutte polynomial which answers Question \ref{question}.

Let $P$ be a polymatroid on $[n]$ and $f$ be its rank function. For convenience, for any $t\in [n]$, let $\alpha_{t}=f([n])-f([n]\setminus \{t\})$, $\beta_{t}=f(\{t\})$ and $T_{t}=\{\alpha_{t}, \alpha_{t}+1,\ldots,\beta_{t}\}$. For any $j\in T_{t}$, define
 $P^{t}_{j}:=\{(a_{1},\ldots,a_{n})\in P\mid a_{t}=j\}$ and its projection
\begin{equation}\label{projection}
\widehat{P}^{t}_{j}:=\{(a_{1},\ldots,a_{t-1},a_{t+1},\ldots,a_{n})\in \mathbb{Z}^{n-1}\mid (a_{1},\ldots,a_{n})\in P^{t}_{j}\}.
\end{equation}
 Here the range $T_{t}$ is chosen such that $P^{t}_{j}$ and $\widehat{P}^{t}_{j}$ are nonempty if and only if $j\in T_{t}$. By the Exchange Axiom, $P^{t}_{j}$ and $\widehat{P}^{t}_{j}$ are polymatroids on $[n]$ and on $[n]\setminus\{t\}$, respectively. We next study the relation between the rank function of $P$ and the rank function of the polymatroid $P^{t}_{j}$.  We first state a basic property of polymatroids. See \cite{Kalman1} for a proof.

\begin{lemma}
\label{change}
Let $P$ be a polymatroid on $[n]$. For any basis $\mathbf{a}\in P$ and any subset $I\subseteq[n]$, if $I\notin\mathcal{I}(\mathbf{a})$, then there are $j\in I$ and $k\in [n]\setminus I$ such that $\mathbf{a}-\mathbf{e}_{k}+\mathbf{e}_{j}\in P$.
\end{lemma}

\begin{proposition}\label{rank function of Pj}
Let $P$ be a polymatroid on $[n]$  and $f$ be its rank function. For some $t\in [n]$, let $\alpha_{t}$, $\beta_{t}$, $T_{t}$ and $\widehat{P}^{t}_{j}$ be defined as above. Let $f^{t}_{j}$ be the rank function of the polymatroid $\widehat{P}^{t}_{j}$. Then for any subset $I\subseteq [n]\setminus \{t\}$, we have $f^{t}_{j}(I)=\min\{f(I), f(I\cup \{t\})-j\}$.
\end{proposition}

\begin{proof}
For any $\mathbf{a}\in P$, we know that $\sum\limits_{i\in I}a_{i}\leq f(I)$ and $\sum\limits_{i\in I}a_{i}+a_{t}\leq f(I\cup \{t\})$. Then for any $j\in T_{t}$, we have that $f^{t}_{j}(I)\leq \min\{f(I), f(I\cup \{t\})-j\}$ as $a_{t}=j$.

We now show that there is a basis $\mathbf{c}\in P^{t}_{j}$ such that $\sum\limits_{i\in I}c_{i}= \min\{f(I), f(I\cup \{t\})-j\}$. Let $\mathbf{b}\in P$ be a basis of $P$ satisfying $\sum\limits_{i\in I}b_{i}=f(I)$ and $\sum\limits_{i\in I}b_{i}+b_{t}=f(I\cup\{t\})$. (Such a basis must exist. For example, it can be taken as the lexicographically maximal basis of $P$ with respect to some order in which $I$ forms the first $|I|$ elements of $[n]$ and  $t$ is the $(|I|+1)$'st.) We separate the following cases.

\begin{enumerate}
\item[(i)] Assume that $f(I)=f(I\cup \{t\})-j$. Then $\mathbf{c}=\mathbf{b}\in P^{t}_{j}$ and it satisfies the required property.

\item[(ii)] Assume that $f(I)< f(I\cup \{t\})-j$. Then $j<f(I\cup \{t\})-f(I)=b_{t}$. Note that $j\geq \alpha_{t}=f([n])-f([n]\setminus \{t\})$. By Lemma \ref{change}, we have that for any $\mathbf{d}\in P$ satisfying $d_{t}>j$, there is $k\in [n]\setminus \{t\}$ such that $\mathbf{d}-\mathbf{e}_{t}+\mathbf{e}_{k}\in P$ as $[n]\setminus \{t\}\notin\mathcal{I}(\mathbf{d})$. This implies that for any $m\in [1,b_{t}-j]$, there is $k_{m}\in [n]\setminus \{t\}$ such that $\mathbf{b}^{m+1}=\mathbf{b}^{m}-\mathbf{e}_{t}+\mathbf{e}_{k_{m}}\in P$, where $\mathbf{b}^{1}=\mathbf{b}$. Note also that $k_{m}\notin I$ as $I\in\mathcal{I}(\mathbf{b}^{m})$ for all $m$. Hence, the basis $\mathbf{c}=\mathbf{b}^{b_{t}-j+1}$ satisfies the required property.

\item[(iii)] Assume that $f(I)> f(I\cup \{t\})-j$. Similarly to (ii), by Lemma \ref{change}, for any $m\in [1,j-b_{t}]$, there is $k'_{m}\in [n]\setminus \{t\}$ such that $\mathbf{c}^{m+1}=\mathbf{c}^{m}+\mathbf{e}_{t}-\mathbf{e}_{k'_{m}}\in P$ as $\{t\}\notin\mathcal{I}(\mathbf{c}^{m})$, where $\mathbf{c}^{1}=\mathbf{b}$. We know that $k'_{m}\notin [n]\setminus (I\cup \{t\})$ as $I\cup \{t\}\in\mathcal{I}(\mathbf{c}^{m})$, that is, $k'_{m}\in I$ for all $m$. This implies that $\mathbf{c}=\mathbf{c}^{j-b_{t}+1}$ satisfies the required property.
\end{enumerate}
The conclusion holds in each case.
\end{proof}

\begin{definition}
Let $P$ be a polymatroid on $[n]$ with rank function $f$. For a subset $A\subseteq [n]$, the \emph{deletion} $P\setminus A$ and \emph{contraction} $P/A$, which are polymatroids on $[n]\setminus A$, are given by the rank functions $f_{P\setminus A}(T)=f(T)$ and $f_{P/A}(T)=f(T\cup A)-f(A)$, for any subset $T\subseteq [n]\setminus A$, respectively.
 \end{definition}

 By the submodularity of $f$, for any subset $I\subseteq [n]\setminus \{t\}$, we have $f(I\cup \{t\})+f([n]\setminus \{t\})\geq f([n])+f(I)$ and $f(\{t\})+f(I)\geq f(I\cup \{t\})$.  Hence, Proposition \ref{rank function of Pj} implies that
 \begin{eqnarray}\label{Del-Con}
 \widehat{P}^{t}_{\alpha_{t}}=P\setminus \{t\}\ \text{and} \ \widehat{P}^{t}_{\beta_{t}}=P/\{t\}.
\end{eqnarray}

The following conclusion is given in Proposition 4.11 (f) of \cite{Bernardi}.

\begin{theorem}\cite{Bernardi}\label{deletion-contraction-1}
Let $P$ be a polymatroid on $[n]$. For some $t\in [n]$,
\begin{enumerate}
\item [(i)] if $\beta_{t}-\alpha_{t}=0$, then $\mathscr{T}_{P}(x,y)=(x+y-1)\mathscr{T}_{\widehat{P}^{t}_{\alpha_{t}}}(x,y)=(x+y-1)\mathscr{T}_{\widehat{P}^{t}_{\beta_{t}}}(x,y)$;
\item[(ii)] if $\beta_{t}-\alpha_{t}=1$, then $\mathscr{T}_{P}(x,y)=x\mathscr{T}_{\widehat{P}^{t}_{\alpha_{t}}}(x,y)+y\mathscr{T}_{\widehat{P}^{t}_{\beta_{t}}}(x,y)$.
\end{enumerate}
\end{theorem}

In fact, this relation generalizes the deletion-contraction formula of the classical Tutte polynomial. In other words, if $M$ is a matroid, and  $P=P(M)$ is its corresponding polymatroid, then Theorem \ref{deletion-contraction-1} is consistent with the deletion-contraction formula of $T_{M}(x,y)$. Theorem \ref{deletion-contraction-1} (i) and the earlier relation \eqref{IX-T} imply that
 \begin{eqnarray}\label{IX-pre}
 I_{\widehat{P}^{t}_{j}}=I_{P^{t}_{j}}\ \text{and}\  X_{\widehat{P}^{t}_{j}}=X_{P^{t}_{j}}\ \text{for any} \ j\in T_{t}.
\end{eqnarray}
That is, when we consider essentially the same $P$ as a polymatroid over a ground set with an extra element, $\mathscr{T}_{P}$ changes but $I_{P}$ and $X_{P}$ do not change. We now present a deletion-contraction formula for the case $\beta_{t}-\alpha_{t}>0$ which generalizes Theorem \ref{deletion-contraction-1} (ii).

\begin{theorem}\label{deletion-contraction-J}
Let $P$ be a polymatroid on $[n]$. For some $t\in [n]$, if $\beta_{t}-\alpha_{t}>0$, then $$\mathscr{T}_{P}(x,y)=x\mathscr{T}_{\widehat{P}^{t}_{\alpha_{t}}}(x,y)+y\mathscr{T}_{\widehat{P}^{t}_{\beta_{t}}}(x,y)+\sum_{j\in T_{t}\setminus \{\alpha_{t}, \beta_{t}\}}\mathscr{T}_{\widehat{P}^{t}_{j}}(x,y).$$
\end{theorem}

\begin{proof}
 By the $S_{n}$-invariance of $\mathscr{T}_{P}(x,y)$, we may let $t=n$. Then we have the following two claims.

\textbf{Claim 1.} If $\beta_{t}-\alpha_{t}>0$, then
\begin{enumerate}
\item [(i)] $t\in \mathrm{Int}_{P}(\mathbf{a})\setminus \mathrm{Ext}_{P}(\mathbf{a})$ for any $\mathbf{a}\in P^{t}_{\alpha_{t}}$;
\item [(ii)] $t\in \mathrm{Ext}_{P}(\mathbf{a})\setminus \mathrm{Int}_{P}(\mathbf{a})$ for any $\mathbf{a}\in P^{t}_{\beta_{t}}$;
\item [(iii)] $t\in [n]\setminus (\mathrm{Ext}_{P}(\mathbf{a}) \cup \mathrm{Int}_{P}(\mathbf{a}))$ for any $\mathbf{a}\in P^{t}_{j}$, where $j\in T_{t}\setminus \{\alpha_{t}, \beta_{t}\}$;
\item [(iv)] $t\in \mathrm{Ext}_{P^{t}_{j}}(\mathbf{a})\cap \mathrm{Int}_{P^{t}_{j}}(\mathbf{a})$ for any $\mathbf{a}\in P^{t}_{j}$,  where $j\in T_{t}$.
\end{enumerate}

\emph{Proof of Claim 1.} For any $\mathbf{a}\in P^{t}_{\alpha_{t}}$, we have $[n]\setminus \{t\}\in \mathcal{I}_{P}(\mathbf{a})$ by the definition of $P^{t}_{\alpha_{t}}$. Then $t\in \mathrm{Int}_{P}(\mathbf{a})$ by Theorem \ref{ia} (i). If $\mathbf{a}\in P^{t}_{\beta_{t}}$, then $t\in \mathrm{Ext}_{P}(\mathbf{a})$  by Theorem \ref{ia} (ii) as $\{t\}\in \mathcal{I}_{P}(\mathbf{a})$.
Conversely, by Lemma \ref{change} and $t=n$ we have $t\notin \mathrm{Int}_{P}(\mathbf{a})$ for $\mathbf a\notin P^{t}_{\alpha_{t}}$ and $t\notin \mathrm{Ext}_{P}(\mathbf{a})$ for $\mathbf a\notin P^{t}_{\beta_{t}}$ because the tightness of $[n]\setminus\{t\}$ and $\{t\}$, respectively, for $\mathbf a$ fails in these cases.
From this, claims (i), (ii), and (iii) all follow.
As to (iv), it is clear that for any $j\in T_{t}$, we have $t\in \mathrm{Int}_{P^{t}_{j}}(\mathbf{a})\cap \mathrm{Ext}_{P^{t}_{j}}(\mathbf{a})$ for any $\mathbf{a}\in P^{t}_{j}$, since $a_{t}=j$, that is, $[n]\setminus \{t\} \in \mathcal{I}_{P^{t}_{j}}(\mathbf{a})$ and $\{t\}\in \mathcal{I}_{P^{t}_{j}}(\mathbf{a})$.

\textbf{Claim 2.} For any $j\in T_{t}$, let $\mathbf{a}\in P^{t}_{j}$. Then for any $k\in [n]\setminus\{t\}$, we have that $k\in \mathrm{Int}_{P}(\mathbf{a})$ if and only if $k\in \mathrm{Int}_{P^{t}_{j}}(\mathbf{a})$, and $k\in \mathrm{Ext}_{P}(\mathbf{a})$ if and only if $k\in \mathrm{Ext}_{P^{t}_{j}}(\mathbf{a})$.

\emph{Proof of Claim 2.} It is easy to see that $\mathrm{Int}_{P}(\mathbf{a})\subseteq \mathrm{Int}_{P^{t}_{j}}(\mathbf{a})$ and $\mathrm{Ext}_{P}(\mathbf{a})\subseteq \mathrm{Ext}_{P^{t}_{j}}(\mathbf{a})$ as $t=n$ and $P^{t}_{j}\subseteq P$.

If $k\notin \mathrm{Int}_{P}(\mathbf{a})$, then there is $k'<k$ such that $\mathbf{b}=\mathbf{a}-\mathbf{e}_{k}+\mathbf{e}_{k'}\in P$. Note that $t=n$. Then $b_{t}=a_{t}=j$, that is, $\mathbf{b}\in P^{t}_{j}$. This implies that $k\notin \mathrm{Int}_{P^{t}_{j}}(\mathbf{a})$. Similarly, if $k\notin \mathrm{Ext}_{P}(\mathbf{a})$, then there is  $k''<k$ such that $\mathbf{c}=\mathbf{a}+\mathbf{e}_{k}-\mathbf{e}_{k''}\in P$. Hence, $k\notin \mathrm{Ext}_{P^{t}_{j}}(\mathbf{a})$ because $\mathbf{c}\in P^{t}_{j}$. Thus, the claim holds.

We know that for any basis $\mathbf{a}\in P$, there is an integer $j\in T_{t}$ such that $\mathbf{a}\in P^{t}_{j}$. Then by Claims 1 and 2,
\small{\begin{eqnarray*}
\mathscr{T}_{P}(x,y) &=&\sum_{\mathbf{a}\in P}x^{oi_{P}(\mathbf{a})}y^{oe_{P}(\mathbf{a})}(x+y-1)^{ie_{P}(\mathbf{a})}\\
&=&\sum_{\mathbf{a}\in P^{t}_{\alpha_{t}}}x^{oi_{P}(\mathbf{a})}y^{oe_{P}(\mathbf{a})}(x+y-1)^{ie_{P}(\mathbf{a})}+\sum_{\mathbf{a}\in P^{t}_{\beta_{t}}}x^{oi_{P}(\mathbf{a})}y^{oe_{P}(\mathbf{a})}(x+y-1)^{ie_{P}(\mathbf{a})}\\
&&+\sum_{j\in T_{t}\setminus \{\alpha_{t}, \beta_{t}\}}\left[\sum_{\mathbf{a}\in P^{t}_{j}}x^{oi_{P}(\mathbf{a})}y^{oe_{P}(\mathbf{a})}(x+y-1)^{ie_{P}(\mathbf{a})}\right]\\
&=&\sum_{\mathbf{a}\in P^{t}_{\alpha_{t}}}x^{oi_{P^{t}_{\alpha_{t}}}(\mathbf{a})+1}y^{oe_{P^{t}_{\alpha_{t}}}(\mathbf{a})}(x+y-1)^{ie_{P^{t}_{\alpha_{t}}}(\mathbf{a})-1}\\
&&+\sum_{\mathbf{a}\in P^{t}_{\beta_{t}}}x^{oi_{P^{t}_{\beta_{t}}}(\mathbf{a})}y^{oe_{P^{t}_{\beta_{t}}}(\mathbf{a})+1}(x+y-1)^{ie_{P^{t}_{\beta_{t}}}(\mathbf{a})-1}\\
&&+\sum_{j\in T_{t}\setminus \{\alpha_{t}, \beta_{t}\}}\left[ \sum_{\mathbf{a}\in P^{t}_{j}}x^{oi_{P^{t}_{j}}(\mathbf{a})}y^{oe_{P^{t}_{j}}(\mathbf{a})}(x+y-1)^{ie_{P^{t}_{j}}(\mathbf{a})-1}\right]\\
&=& (x+y-1)^{-1}\left[ x\mathscr{T}_{P^{t}_{\alpha_{t}}}(x,y)+y\mathscr{T}_{P^{t}_{\beta_{t}}}(x,y)+\sum_{j\in T_{t}\setminus \{\alpha_{t}, \beta_{t}\}}\mathscr{T}_{P^{t}_{j}}(x,y)\right].
\end{eqnarray*}}

Hence, $\mathscr{T}_{P}(x,y)=x\mathscr{T}_{\widehat{P}^{t}_{\alpha_{t}}}(x,y)+y\mathscr{T}_{\widehat{P}^{t}_{\beta_{t}}}(x,y)+\sum\limits_{j\in T_{t}\setminus \{\alpha_{t}, \beta_{t}\}}\mathscr{T}_{\widehat{P}^{t}_{j}}(x,y)$ since $\mathscr{T}_{\widehat{P}^{t}_{j}}(x,y)=(x+y-1)^{-1}\mathscr{T}_{P^{t}_{j}}(x,y)$ for any $j\in T_{t}$ by Theorem \ref{deletion-contraction-1} (i).
\end{proof}

Note that if $P$ is a hypergraphical polymatroid, then both $\widehat{P}^{t}_{\alpha_{t}}$ and $\widehat{P}^{t}_{\beta_{t}}$ are hypergraphical. However,  $\widehat{P}^{t}_{j}$ is not necessarily hypergraphical for $j\in T_{t}\setminus \{\alpha_{t}, \beta_{t}\}$.

A deletion-contraction formula of interior  and exterior polynomials of polymatroids can be obtained from Theorems \ref{deletion-contraction-1} (i) and \ref{deletion-contraction-J}.

\begin{corollary} \label{deletion-contraction-IX}
Let $P$ be a polymatroid on $[n]$. Then for any $t\in [n]$, we have
\begin{enumerate}
\item [(i)]  $I_{P}(x)=I_{\widehat{P}^{t}_{\alpha_{t}}}(x)+x\sum_{j\in T_{t}\setminus \{\alpha_{t}\}}I_{\widehat{P}^{t}_{j}}(x)$;
\item[(ii)]  $X_{P}(y)=X_{\widehat{P}^{t}_{\beta_{t}}}(y)+y\sum_{j\in T_{t}\setminus \{\beta_{t}\}}X_{\widehat{P}^{t}_{j}}(y)$.
\end{enumerate}
\end{corollary}

\section{Monotonicity properties}
\label{sec4}

In this section, we prove two monotonicity properties of interior and exterior polynomials of polymatroids. For two polynomials $g_{1}$ and $g_{2}$, we say $g_{1}\leq g_{2}$ if each coefficient of $g_{1}$ is less than or equal to the corresponding coefficient of $g_{2}$.

We first consider the case when there is an inclusion of the two polymatroids.

\begin{theorem}\label{subset-monotonicity}
Let $P$ and $P'$ be two polymatroids on the same ground set $[n]$. If $P'\subseteq P$, then $I_{P'}\leq I_{P}$ and $X_{P'}\leq X_{P}$.
\end{theorem}

\begin{proof}
It is enough to prove the statement for the exterior polynomial because of \eqref{duality} and because $P'\subseteq P$ implies $-P'\subseteq -P$. We now prove the claim by induction on $n$.

 Note that the constant term  of the exterior  polynomial of any polymatroid is 1. Then the statement is clear in the case $n=1$. Assume that the conclusion holds for $n\leq m-1$. Let $f$ and $f'$ be the rank functions of $P$ and $P'$, respectively.  Then for $n=m$, we divide the proof into two cases.

\textbf{Case 1.} Assume that there is some $t\in [n]$ such that $f'(\{t\})=f(\{t\})$. Let  $T_{t}=\{f([n])-f([n]\setminus\{t\}),\ldots,f(\{t\})\}$ (same as the definition of $T_{t}$ in Section 3) and $T'_{t}=\{f'([n])-f'([n]\setminus\{t\}),\ldots,f'(\{t\})\}$. Note that $f'([n])=f([n])$ and by \eqref{RF}, we  have $f'(I)\leq f(I)$ for all $I\subseteq [n]$ as $P'\subseteq P$. Then  $T'_{t}\subseteq T_{t}$ and $\widehat{P}'^{t}_{j}\subseteq \widehat{P}^{t}_{j}$ (cf. \eqref{projection}) for any $j\in T'_{t}$.
By the induction hypothesis and Corollary \ref{deletion-contraction-IX}, we know that
\begin{eqnarray*}
X_{P}(y)&=&X_{\widehat{P}^{t}_{f(\{t\})}}(y)+y\sum_{j\in T_{t}\setminus \{f(\{t\})\}}X_{\widehat{P}^{t}_{j}}(y)\\
&=&X_{\widehat{P}^{t}_{f(\{t\})}}(y)+y\sum_{j\in T_{t}'\setminus \{f(\{t\})\}}X_{\widehat{P}^{t}_{j}}(y)+y\sum_{j\in T_{t}\setminus T'_{t}}X_{\widehat{P}^{t}_{j}}(y)\\
&\geq&X_{\widehat{P}'^{t}_{f'(\{t\})}}(y)+y\sum_{j\in T_{t}'\setminus \{f'(\{t\})\}}X_{\widehat{P}'^{t}_{j}}(y)+y\sum_{j\in T_{t}\setminus T'_{t}}X_{\widehat{P}^{t}_{j}}(y)\\
&=&X_{P'}(y)+y\sum_{j\in T_{t}\setminus T'_{t}}X_{\widehat{P}^{t}_{j}}(y)\\
&\geq& X_{P'}(y).
\end{eqnarray*}

\textbf{Case 2.} Assume that $f'(\{t'\})<f(\{t'\})$ for all $t'\in [n]$. Then choose an arbitrary index $t\in [n]$. By the Exchange Axiom, the set $P''=\{\mathbf{a}\in P\mid a_{t}\leq f'(\{t\})\}$ is also a polymatroid.  Case 1 implies that $X_{P'}\leq X_{P''}$. We now prove that $X_{P''}\leq X_{P}$. This in fact follows from Case 1 applied to any coordinate other than $t$, but we will follow a different route.

By the $S_{n}$-invariance of the exterior polynomial, we may assume that $t=1$. Then
 \begin{eqnarray}\label{EX-M}
\mathrm{Ext}_{P}(\mathbf{b})=\mathrm{Ext}_{P''}(\mathbf{b}) \ \text{for any} \ \mathbf{b}\in P''.
\end{eqnarray}
\begin{enumerate}
\item [(i)] It is obvious that $\mathrm{Ext}_{P}(\mathbf{b})\subseteq \mathrm{Ext}_{P''}(\mathbf{b})$ as $P''\subseteq P$.
\item [(ii)] If $i\notin \mathrm{Ext}_{P}(\mathbf{b})$, then by definition, there exists some $i'\in [n]$ with $i'<i$ such that $\mathbf{c}=\mathbf{b}-\mathbf{e}_{i'}+\mathbf{e}_{i}\in P$. Note that $\mathbf{c}\in P''$ as $c_{t}\leq b_{t}\leq f'(\{t\})$. Hence, $i\notin \mathrm{Ext}_{P''}(\mathbf{b})$.
\end{enumerate}
The equation \eqref{EX-M} implies that $\overline{\epsilon}_{P}(\mathbf{b})=\overline{\epsilon}_{P''}(\mathbf{b})$ for any $\mathbf{b}\in P''$.
Then
 \begin{eqnarray*}
X_{P}(y)&=&\sum_{\mathbf{b}\in P} y^{\overline{\epsilon}_{P}(\mathbf{b})}\\
&=&\sum_{\mathbf{b}\in P''} y^{\overline{\epsilon}_{P}(\mathbf{b})}+\sum_{\mathbf{b}\in P\setminus P''} y^{\overline{\epsilon}_{P}(\mathbf{b})}\\
&=&\sum_{\mathbf{b}\in P''} y^{\overline{\epsilon}_{P''}(\mathbf{b})}+\sum_{\mathbf{b}\in P\setminus P''} y^{\overline{\epsilon}_{P}(\mathbf{b})}\\
&=&X_{P''}(y)+\sum_{\mathbf{b}\in P\setminus P''} y^{\overline{\epsilon}_{P}(\mathbf{b})}\\
&\geq&X_{P''}(y).
\end{eqnarray*}
This completes the proof.
\end{proof}

We next consider the case of minors. We start with
recalling
an easy result (leaving the proof to the reader since it is both classical and trivial to prove).

\begin{proposition}\label{well-define1}
Let $P$ be a polymatroid on $[n]$. Then, for any disjoint subsets $A$ and $B$ of $[n]$, we have $P/ A / B=P/ B / A=P/ (A\cup B)$ and $P\setminus A \setminus B=P\setminus B \setminus A=P\setminus (A\cup B)$, as well as $P\setminus A/B=P/B\setminus A$.
\end{proposition}

\begin{definition}
Let $P$ be a polymatroid on $[n]$.
A polymatroid $P'$ is a \emph{minor} of $P$ if $P'=(P\setminus A)/B=(P/B)\setminus A$ for some disjoint subsets $A$ and $B$ of $[n]$.
\end{definition}

For a polymatroid $P$, let $f_{P}$ and $f_{-P}$ be the rank functions of the polymatroids $P$ and $-P$, respectively. Then by \eqref{RF} and the definition of the dual polymatroid, we have that $f_{-P}(T)=f_{P}([n]\setminus T)-f_{P}([n])$ for any subset $T\subseteq [n]$.
The following relation of duality with contraction and deletion holds.

\begin{proposition}\label{well-dual}
Let $P$ be a polymatroid on $[n]$. Then $-(P\setminus A)=(-P)/A$ and $-(P/ A)= (-P)\setminus A$ for any subset $A\subseteq [n]$.
\end{proposition}

\begin{proof}
Let $f_{P}$, $f_{P\setminus A}$, $f_{-(P\setminus A)}$, $f_{(-P)/A}$ and $f_{-P}$ be the rank functions of the polymatroids $P$, $P\setminus A$, $-(P\setminus A)$, $(-P)/A$ and $-P$, respectively. Then for any subset $T\subseteq [n]\setminus A$, we have
\begin{eqnarray*}
f_{-(P\setminus A)}(T)&=&f_{P\setminus A}([n]\setminus A\setminus T)-f_{P\setminus A}([n]\setminus A)\\
&=&f_{P}([n]\setminus A\setminus T)-f_{P}([n]\setminus A)\\
&=&(f_{P}([n]\setminus (A\cup T))-f_{P}([n]))-(f_{P}([n]\setminus A)-f_{P}([n]))\\
&=&f_{-P}(T\cup A)-f_{-P}(A)\\
&=&f_{(-P)/A}(T).
\end{eqnarray*}
Hence, $-(P\setminus A)=(-P)/A$. This implies that $-((-P)\setminus A)=P/A$. Therefore 
$-(P/ A)= (-P)\setminus A$ also holds.
\end{proof}

We study  monotonicity properties of deletion and contraction before our second main claim.

\begin{lemma}\label{mono}
Let $P$ be a polymatroid on $[n]$. Then $I_{P\setminus A}\leq I_{P}$,  $I_{P/A}\leq I_{P}$, $X_{P\setminus A}\leq X_{P}$ and $X_{P/A}\leq X_{P}$ for any subset $A\subseteq [n]$.
\end{lemma}

\begin{proof}
It is enough to prove that the statements $I_{P\setminus A}\leq I_{P}$ and  $I_{P/A}\leq I_{P}$ hold because of \eqref{duality} and Proposition \ref{well-dual}. By Proposition \ref{well-define1}, it suffices to consider the case that $A$ is a singleton set $\{t\}$. The claims $I_{P\setminus \{t\}}\leq I_{P}$ and $I_{P/\{t\}}\leq I_{P}$, by \eqref{Del-Con} and \eqref{IX-pre},
are equivalent to $I_{P^{t}_{\alpha_{t}}}\leq I_{P}$ and $I_{P^{t}_{\beta_{t}}}\leq I_{P}$, respectively. These follow from Theorem \ref{subset-monotonicity}.
\end{proof}

\begin{theorem}\label{minor-monotonicity}
If $P'$ is a minor of a polymatroid $P$,  then $I_{P'}\leq I_{P}$ and $X_{P'}\leq X_{P}$.
\end{theorem}

\begin{proof}
Without loss of generality, we may assume that $P$ is a polymatroid on $[n]$ and $P'=P\setminus A/B$ for disjoint subsets $A\subseteq [n]$ and $B\subseteq [n]$.  Then $I_{P'}\leq I_{P\setminus A}\leq I_{P}$ and $X_{P'}\leq X_{P\setminus A}\leq X_{P}$ by Lemma \ref{mono}.
\end{proof}

Now let us focus on the hypergraphical cases. For brevity, we replace $I_{P_{\mathcal{H}}}$ and $X_{P_{\mathcal{H}}}$ with $I_{\mathcal{H}}$ and $X_{\mathcal{H}}$ for a hypergraph $\mathcal{H}$.

\begin{corollary}\label{HG-monotonicity}
Let $\mathcal{H}=(V,E)$ and $\mathcal{H}'=(V',E')$ be two hypergraphs. Let $\mathrm{Bip} \mathcal{H}$ and $\mathrm{Bip} \mathcal{H}'$ be the associated bipartite graphs of $\mathcal{H}$ and $\mathcal{H}'$, respectively. If $\mathrm{Bip} \mathcal{H}'$ is a subgraph of $\mathrm{Bip} \mathcal{H}$, then $I_{\mathcal{H}'}\leq I_{\mathcal{H}}$.
If $E'\subseteq E$ and $V'\subseteq V$ hold as well, then $X_{\mathcal{H}'}\leq X_{\mathcal{H}}$.
\end{corollary}

We phrased the claim as above to underline that the interior polynomial is transpose-invariant \cite{Kalman3} (i.e., $I_{\mathcal{H}}$ depends only on $\mathrm{Bip} \mathcal{H}$), whereas the exterior polynomial is not.

\begin{proof}
\begin{figure}
\centering
\includegraphics[width=0.8\textwidth]{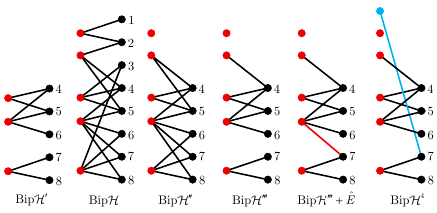}
\caption{An example  in Corollary \ref{HG-monotonicity}}
\label{monotonicity1}
\end{figure}
  See Figure \ref{monotonicity1} for an illustration. We may assume that $E'\subseteq E$ and $V'\subseteq V$.

Let $\mathrm{Bip} \mathcal{H}''$ be the bipartite graph obtained  from $\mathrm{Bip} \mathcal{H}$ by removing all vertices of $E\setminus E'$ and all edges incident with elements of $E\setminus E'$. Let $\mathcal{H}''$ be the hypergraph induced by $\mathrm{Bip} \mathcal{H}''$ with $E'$ as the set of hyperedges. Then $I_{\mathcal{H}''}\leq I_{\mathcal{H}}$ and $X_{\mathcal{H}''}\leq X_{\mathcal{H}}$ by Theorem \ref{minor-monotonicity}.

Assume that there are $c$ isolated vertices of $V\setminus V'$ in $\mathrm{Bip} \mathcal{H}''$. Let $\mathrm{Bip} \mathcal{H}'''$ be a bipartite graph obtained  from $\mathrm{Bip} \mathcal{H}'$ by adding $|V\setminus V'|$ vertices to $V'$ and $|V\setminus V'|-c$ edges such that for any non-isolated vertex $v'$ of $V\setminus V'$ in $\mathrm{Bip} \mathcal{H}''$, there exists exactly one edge $uv'$ in $\mathrm{Bip} \mathcal{H}'''$, where $uv'$ is an edge of $\mathrm{Bip} \mathcal{H}''$.
 Let $c(\mathrm{Bip} \mathcal{H}'')$ and $c(\mathrm{Bip} \mathcal{H}''')$ be the numbers of the connected components of $\mathrm{Bip} \mathcal{H}''$ and $\mathrm{Bip} \mathcal{H}'''$, respectively. Note that $\mathrm{Bip} \mathcal{H}'''$ is a spanning subgraph of $\mathrm{Bip} \mathcal{H}''$. Then $c(\mathrm{Bip} \mathcal{H}''')\geq c(\mathrm{Bip} \mathcal{H}'')$. Put $c'=c(\mathrm{Bip} \mathcal{H}''')-c(\mathrm{Bip} \mathcal{H}'')$.
Arbitrarily choose a set of $c'$ edges $\widehat{E}$ in $\mathrm{Bip} \mathcal{H}''$ such that $c(\mathrm{Bip} \mathcal{H}'''+\widehat{E})=c(\mathrm{Bip} \mathcal{H}'')$.
Let $\mathrm{Bip} \mathcal{H}^{4}$ be a bipartite graph obtained from $\mathrm{Bip} \mathcal{H}'''+\widehat{E}$ by adding one vertex $v'$ to $V$ and replacing $uv$ by $uv'$ for each edge $uv\in \widehat{E}$, where $u\in E'$ and $v\in V$. Let $\mathcal{H}^{4}$ be the hypergraph induced by $\mathrm{Bip} \mathcal{H}^{4}$ with $E'$ as the set of hyperedges.

Recall that the elements of a hypergraphical polymatroid $P_{\mathcal{H}}$ are hypertrees, which in turn are essentially degree distributions of maximal spanning forests of $\mathrm{Bip} \mathcal{H}$.  It is easy to see that $\mathrm{Bip} \mathcal{H}'''+\widehat{E}$ is a spanning subgraph of $\mathrm{Bip} \mathcal{H}''$. Recall that $c(\mathrm{Bip} \mathcal{H}'''+\widehat{E})= c(\mathrm{Bip} \mathcal{H}'')$. We have that each maximal spanning forest of $\mathrm{Bip} \mathcal{H}'''+\widehat{E}$ is a maximal spanning forest of $\mathrm{Bip} \mathcal{H}''$. Moreover, there is a one-to-one correspondence between maximal spanning forests of $\mathrm{Bip} \mathcal{H}'''+\widehat{E}$ and maximal spanning forests of $\mathrm{Bip} \mathcal{H}^{4}$. For each corresponding pair, their degrees at
each $e\in E'$
are the same. This implies that the hypergraphical polymatroid $P_{\mathcal{H}^{4}}$ is a subset of $P_{\mathcal{H}''}$.
Then by Theorem \ref{subset-monotonicity}, we know that $I_{\mathcal{H}^{4}}\leq I_{\mathcal{H}''}$ and $X_{\mathcal{H}^{4}}\leq X_{\mathcal{H}''}$.
  Moreover, by Theorem \ref{invariants} (i), we have that $I_{\mathcal{H}'}=I_{\mathcal{H}^{4}}$ and $X_{\mathcal{H}'}=X_{\mathcal{H}^{4}}$. Hence, the conclusion holds.
\end{proof}

\begin{remark} \label{counterexample}
Let $P$ and $P'$ be polymatroids on the same set $[n]$. If $P'\subseteq P$, then it is possible that there are $i,j\in \mathbb{Z}$ such that $[x^{i}y^{j}]\mathscr{T}_{P'}(x,y)> [x^{i}y^{j}]\mathscr{T}_{P}(x,y)$. Likewise, the monotonicity property of the polymatroid Tutte polynomial does not hold for minors either.
 Counterexamples can already be found among hypergraphical cases. Furthermore, for a polymatroid $P$ over $[n]$ with rank $d$, let us define $$T_{P}(x,y):=\frac{(x+y-xy)^{n}}{x^{n-d}y^{d}}\mathscr{T}_{P}\left(\frac{x}{x+y-xy},\frac{y}{x+y-xy}\right).$$ This is an equivalent form of $\mathscr{T}_{P}(x,y)$ and when written in this way, the invariant becomes a direct generalization of the Tutte polynomial of matroids, cf.\ \eqref{TM}. We claim that $T_{P}(x,y)$ does not satisfy the monotonicity property, either.

\begin{figure}
\centering
\includegraphics[width=0.8\textwidth]{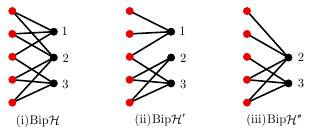}
\caption{$\text{Bip} \mathcal{H}$,  $\text{Bip}\mathcal{H}'$  and $\text{Bip} \mathcal{H}''$ in Remark \ref{counterexample}}
\label{1}
\end{figure}

For example, let us consider the hypergraphical polymatroids $P_{\mathcal{H}}$ and $P_{\mathcal{H}'}$ over $\{1,2,3\}$, where the hypergraphs $\mathcal{H}$ and $\mathcal{H}'$ are defined by their associated bipartite graphs in Figure \ref{1} (i) and (ii), respectively. Let also $P_{\mathcal{H}''}=P_{\mathcal{H}}\setminus\{1\}$, see Figure \ref{1} (iii) for the associated bipartite graph. Then $P_{\mathcal{H}'}\subseteq P_{\mathcal{H}}$ and $P_{\mathcal{H}''}$ is a minor of $P_{\mathcal{H}}$. On the other hand,
$$\mathscr{T}_{P_{\mathcal{H}}}(x,y)=x^{3}+3x^{2}y+3xy^{2}+y^{3}+2x^{2}+3xy+y^{2}-x-2,$$  $$\mathscr{T}_{P_{\mathcal{H}'}}(x,y)=x^{3}+3x^{2}y+3xy^{2}+y^{3}-2x^{2}-4xy-2y^{2}+x+y,$$  and
$$\mathscr{T}_{P_{\mathcal{H}''}}(x,y)=x^{2}+2xy+y^{2}-x-y.$$
In particular,
$[xy^{0}]\mathscr{T}_{P_{\mathcal{H}'}}(x,y)> [xy^{0}]\mathscr{T}_{P_{\mathcal{H}}}(x,y)$, $[x^{0}y]\mathscr{T}_{P_{\mathcal{H}'}}(x,y)> [x^{0}y]\mathscr{T}_{P_{\mathcal{H}}}(x,y)$, \\
$[x^{0}y^{0}]\mathscr{T}_{P_{\mathcal{H}'}}(x,y)> [x^{0}y^{0}]\mathscr{T}_{P_{\mathcal{H}}}(x,y)$ and $[x^{0}y^{0}]\mathscr{T}_{P_{\mathcal{H}''}}(x,y)>[x^{0}y^{0}]\mathscr{T}_{P_{\mathcal{H}}}(x,y)$.

Note that $\mathscr{T}_{P_{\mathcal{H}}}(x,y)-\mathscr{T}_{P_{\mathcal{H'}}}(x,y)=4x^{2}+7xy+3y^{2}-2x-y-2.$
This yields
\begin{eqnarray*}
&&T_{P_{\mathcal{H}}}(x,y)-T_{P_{\mathcal{H}'}}(x,y)\\
&=&\left(\mathscr{T}_{P_{\mathcal{H}}}(\frac{x}{x+y-xy},\frac{y}{x+y-xy})-\mathscr{T}_{P'_{\mathcal{H'}}}(\frac{x}{x+y-xy},\frac{y}{x+y-xy})\right)\cdot\frac{(x+y-xy)^{3}}{x^{-1}y^{4}}\\
&=&xy^{-4}(2x^{3}y^{3}-8x^{3}y^{2}-7x^{2}y^{3}+11x^{2}y^{2}+6x^{3}y+5xy^{3}).
\end{eqnarray*}
Hence, we have $[x^{4}y^{-2}]T_{P_{\mathcal{H}'}}(x,y)> [x^{4}y^{-2}]T_{P_{\mathcal{H}}}(x,y)$ and $[x^{3}y^{-1}]T_{P_{\mathcal{H}''}}(x,y)>$\\ $[x^{3}y^{-1}]T_{P_{\mathcal{H}}}(x,y)$.

As to $\mathcal{H}''$, whose ground set differs in size from that of $\mathcal{H}$, we have $$T_{P_{\mathcal{H}''}}(x,y)=x^{2}y^{-4}(x^{2}+2xy+y^{2}-(x+y)(x+y-xy)).$$ In particular, $[x^{4}y^{-2}]T_{P_{\mathcal{H}''}}(x,y)=0$. On the other hand,
\begin{eqnarray*}
T_{P_{\mathcal{H}}}(x,y)&=&xy^{-4}(x^{3}+3x^{2}y+3xy^{2}+y^{3}+(2x^{2}+3xy+y^{2})(x+y-xy)\\
&&-x(x+y-xy)^{2}-2(x+y-xy)^{3})
\end{eqnarray*}
has $[x^{4}y^{-2}]T_{P_{\mathcal{H}}}(x,y)=-7<0=[x^{4}y^{-2}]T_{P_{\mathcal{H}''}}(x,y)$.
\end{remark}

\section{Further applications of the deletion-contraction formula}
\label{sec5}

In this section, by applying the deletion-contraction formula of Section 3, we compute the coefficients of some terms of the polymatroid Tutte polynomial and characterize the class of hypergraphs  which are such that each coefficient in the exterior polynomial attains its maximal value.

\subsection{The extremal coefficients of $\mathscr{T}_{P}(x,y)$}
\label{sec5.1}
In this subsection, we study the coefficients of some extremal terms of the polymatroid Tutte polynomial.
We start with results for coefficients from \cite{Bernardi, Guan3}.

\begin{lemma} \cite{Bernardi}\label{coe-degree-n}
Let $P$ be a polymatroid over $[n]$. Then for any $k\in[n]\cup\{0\}$, $$[x^{n-k}y^{k}]\mathscr{T}_{P}(x,y)=\binom{n}{k}.$$
\end{lemma}

\begin{lemma} \cite{Guan3} \label{n-1-coe}
Let $P$ be a polymatroid on $[n]$ and let $f$ be its rank function. Then the coefficients of terms of degree $n-1$ for $\mathscr{T}_{P}(x,1)$ and $\mathscr{T}_{P}(1,y)$ are $\sum\limits_{i\in [n]}f(\{i\})-f([n])$ and $\sum\limits_{i\in [n]}f([n]\setminus \{i\})-(n-1)f([n])$, respectively.
\end{lemma}

By Lemmas \ref{coe-degree-n} and  \ref{n-1-coe}, we have the following statement.

\begin{lemma}\label{coe-degree-n-1}
Let $P$ be a polymatroid over $[n]$ with rank function $f$. Then
\begin{enumerate}
\item [(i)]  $[x^{n-1}y^{0}]\mathscr{T}_{P}(x,y)=\sum_{i\in [n]}f(\{i\})-f([n])-n$;
\item [(ii)] $[x^{0}y^{n-1}]\mathscr{T}_{P}(x,y)=\sum_{i\in [n]}f([n]\setminus \{i\})-(n-1)f([n])-n$.
\end{enumerate}
\end{lemma}

We next study the coefficients of degree $n-1$ terms in $\mathscr{T}_{P}(x,y)$ and obtain a generalization of
Lemma \ref{coe-degree-n-1}, using the deletion-contraction relation of Section \ref{sec3}.

\begin{theorem} \label{coe-degree-n-1'}
Let $P$ be a polymatroid on $[n]$ with rank function $f$. Then for any $k\in[n]$,
\begin{equation}\label{degree n-1}
[x^{n-k}y^{k-1}]\mathscr{T}_{P}(x,y)
=\sum_{\stackrel{S\subseteq [n]}{|S|=k-1}}f(S)+\sum_{\stackrel{S'\subseteq [n]}{|S'|=k}}f(S')-\binom{n}{k-1}f([n])-k\binom{n}{k}.
\end{equation}
\end{theorem}

\begin{proof}
By Lemma \ref{coe-degree-n-1} (ii), the conclusion is true for $k=n$. We next assume $k\in [n-1]$.
For any $t\in [n]$,  let $\alpha_{t}=f([n])-f([n]\setminus \{t\})$ and $\beta_{t}=f(\{t\})$.

If $\alpha_{t'}=\beta_{t'}$ for all $t'\in [n]$, then $P$ has a unique basis. Hence,
$\mathscr{T}_{P}(x,y)=(x+y-1)^{n}$.
Then the left-hand side of  equation \eqref{degree n-1} is
$$[x^{n-k}y^{k-1}]\mathscr{T}_{P}(x,y)=-\binom{n}{n-1}\binom{n-1}{k-1}
=-n\binom{n-1}{k-1}.$$
Moreover,
the right-hand side of equation \eqref{degree n-1} is $$\sum_{\stackrel{S\subseteq [n]}{|S|=k-1}}f(S)+\sum_{\stackrel{S'\subseteq [n]}{|S'|=k}}f(S')-\binom{n}{k-1}f([n])-k\binom{n}{k}=-k\binom{n}{k}=-n\binom{n-1}{k-1}.$$

Assume that there exists some $t\in [n]$ satisfying $\alpha_{t}\neq\beta_{t}$. Then we prove  equation \eqref{degree n-1} by induction on $n$. It is obvious in the case $n=1$.
Suppose that $n\geq2$. Let $T_{t}=\{\alpha_{t}, \alpha_{t}+1,\ldots,\beta_{t}\}$. For any $j\in T_{t}$, let $P^{t}_{j}=\{(a_{1},\ldots,a_{n})\in P\mid a_{t}=j\}$ and $\widehat{P}^{t}_{j}=\{(a_{1},\ldots,a_{t-1},a_{t+1},\ldots,a_{n})\in \mathbb{Z}^{n-1}\mid (a_{1},\ldots,a_{n})\in P^{t}_{j}\}$, furthermore let $f^{t}_{j}$ be the rank function of $\widehat{P}^{t}_{j}$.

By Proposition \ref{rank function of Pj}, for any subset $I\subseteq[n]\setminus \{t\}$, we see that $f^{t}_{\alpha_{t}}(I)=f(I)$ and $f^{t}_{\beta_{t}}(I)=f(I\cup \{t\})-f(\{t\})$.
 Then by the induction hypothesis, for any $k\in[n-1]$,
\begin{eqnarray*}
  &&[x^{n-k-1}y^{k-1}]\mathscr{T}_{P^{t}_{\alpha_{t}}}(x,y) \\
  &=&\sum_{\stackrel{T\subseteq [n]\setminus \{t\}}{|T|=k-1}}f^{t}_{\alpha_{t}}(T)+\sum_{\stackrel{T'\subseteq [n]\setminus \{t\}}{|T'|=k}}f^{t}_{\alpha_{t}}(T')-\binom{n-1}{k-1}f^{t}_{\alpha_{t}}([n]\setminus \{t\})-k\binom{n-1}{k}\\
&=&\sum_{\stackrel{T\subseteq [n]\setminus \{t\}}{|T|=k-1}}f(T)+\sum_{\stackrel{T'\subseteq [n]\setminus \{t\}}{|T'|=k}}f(T')-\binom{n-1}{k-1}f([n]\setminus \{t\})-k\binom{n-1}{k}\\
\end{eqnarray*}
 and
\begin{eqnarray*}
 && [x^{n-k}y^{k-2}]\mathscr{T}_{P^{t}_{\beta_{t}}}(x,y)\\
 &=&\sum_{\stackrel{T\subseteq [n]\setminus \{t\}}{|T|=k-2}}f^{t}_{\beta_{t}}(T)+\sum_{\stackrel{T'\subseteq [n]\setminus \{t\}}{|T'|=k-1}}f^{t}_{\beta_{t}}(T')-\binom{n-1}{k-2}f^{t}_{\beta_{t}}([n]\setminus \{t\})-(k-1)\binom{n-1}{k-1}\\
&=&\sum_{\stackrel{T\subseteq [n]\setminus \{t\}}{|T|=k-2}}(f(T\cup\{t\})-f(\{t\}))+\sum_{\stackrel{T\subseteq [n]\setminus \{t\}}{|T'|=k-1}}(f(T'\cup\{t\})-f(\{t\}))\\
  &&-\binom{n-1}{k-2}(f([n])-f(\{t\}))-(k-1)\binom{n-1}{k-1}\\
&=&\sum_{\stackrel{T\subseteq [n]\setminus \{t\}}{|T|=k-2}}f(T\cup\{t\})+\sum_{\stackrel{T\subseteq [n]\setminus \{t\}}{|T'|=k-1}}f(T'\cup\{t\})-\binom{n-1}{k-1}f(\{t\})\\
  &&-\binom{n-1}{k-2}f([n])-(k-1)\binom{n-1}{k-1}.
\end{eqnarray*}
By Lemma \ref{coe-degree-n},
\begin{eqnarray*}
 \sum_{j\in T_{t}\setminus \{\alpha_{t},\beta_{t}\}}[x^{n-k}y^{k-1}]\mathscr{T}_{P^{t}_{j}}(x,y)&=&\sum_{j\in T_{t}\setminus \{\alpha_{t},\beta_{t}\}}\binom{n-1}{k-1}\\
&=&\binom{n-1}{k-1}(\beta_{t}-\alpha_{t}-1)\\
&=&\binom{n-1}{k-1}(f(\{t\})+f([n]\setminus \{t\})-f([n])-1).
\end{eqnarray*}

Note that $$\sum_{\stackrel{T\subseteq [n]\setminus \{t\}}{|T|=k-1}}f(T)+\sum_{\stackrel{T\subseteq [n]\setminus \{t\}}{|T|=k-2}}f(T\cup\{t\})=\sum_{\stackrel{S\subseteq [n]}{|S|=k-1}}f(S)$$ and
$$\sum_{\stackrel{T'\subseteq [n]\setminus \{t\}}{|T'|=k}}f(T')+\sum_{\stackrel{T'\subseteq [n]\setminus \{t\}}{|T'|=k-1}}f(T'\cup\{t\})=\sum_{\stackrel{S'\subseteq [n]}{|S'|=k}}f(S').$$
By Theorem \ref{deletion-contraction-J},
\begin{multline*}
[x^{n-k}y^{k-1}]\mathscr{T}_{P}(x,y)=[x^{n-k-1}y^{k-1}]\mathscr{T}_{\widehat{P}^{t}_{\alpha_{t}}}(x,y)+[x^{n-k}y^{k-2}]\mathscr{T}_{\widehat{P}^{t}_{\beta_{t}}}(x,y)
\\+\sum_{j\in T_{t}\setminus \{\alpha_{t},\beta_{t}\}}[x^{n-k}y^{k-1}]\mathscr{T}_{\widehat{P}^{t}_{j}}(x,y).
\end{multline*}
Hence, equation \eqref{degree n-1} is valid.
\end{proof}

Similarly to Theorem \ref{coe-degree-n-1'}, we can get formulas for $[x^{n-2}y^0]\mathscr{T}_{P}(x,y)$ and $[x^0y^{n-2}]\mathscr{T}_{P}(x,y)$.

\begin{theorem} \label{coe-degree-n-2}
Let $P$ be a polymatroid on $[n]$  with rank function $f$. Then
\begin{enumerate}
\item [(i)]  $[x^{n-2}y^0]\mathscr{T}_{P}(x,y)=\binom{-f([n])+1-n+\sum\limits_{i\in [n]}f(\{i\})}{2}-\left(\sum\limits_{\{i,j\}\subseteq [n]}\binom{f(\{i\})+f(\{j\})-f(\{i,j\})}{2}\right)$;
\item [(ii)] $[x^0y^{n-2}]\mathscr{T}_{P}(x,y)=\binom{-(n-1)f([n])+1-n+\sum\limits_{i\in [n]}f([n]\setminus \{i\})}{2}\\-\left(\sum\limits_{\{i,j\}\subseteq [n]}\binom{f([n]\setminus\{i\})+f([n]\setminus\{j\})-f([n]\setminus\{i,j\})-f([n])}{2}\right)$.
\end{enumerate}
\end{theorem}

\begin{proof}
Similarly to the proof of Theorem \ref{coe-degree-n-1'},
we divide the argument into two cases.

Case 1: $\alpha_{t'}=\beta_{t'}$ for all $t'\in [n]$. In this case, it is simple to check both sides of each equation.

Case 2: There exists some $t\in [n]$ satisfying $\alpha_{t}\neq\beta_{t}$. In this case, we again proceed by induction on $n$. The rest is merely a matter of calculation, which we leave
to the reader.
\end{proof}

Lemma \ref{coe-degree-n}, Theorems \ref{coe-degree-n-1'} and \ref{coe-degree-n-2} imply the following conclusion.

\begin{corollary} \label{n-2-coe}
For a polymatroid $P$ on $[n]$ with rank function $f$,
\begin{enumerate}
\item [(i)]  $[x^{n-2}]\mathscr{T}_{P}(x,1)=\binom{-f([n])+1+\sum\limits_{i\in [n]}f(\{i\})}{2}-\left(\sum\limits_{\{i,j\}\subseteq [n]}\binom{f(\{i\})+f(\{j\})-f(\{i,j\})+1}{2}\right)$;
\item [(ii)] $[y^{n-2}]\mathscr{T}_{P}(1,y)=\binom{-(n-1)f([n])+1+\sum\limits_{i\in [n]}f([n]\setminus \{i\})}{2}\\-\left(\sum\limits_{\{i,j\}\subseteq [n]}\binom{f([n]\setminus\{i\})+f([n]\setminus\{j\})-f([n]\setminus\{i,j\})-f([n])+1}{2}\right)$.
\end{enumerate}
\end{corollary}

\begin{proof}
Note that
$[x^{n-2}]\mathscr{T}_{P}(x,1)=[x^{n-2}y^0]\mathscr{T}_{P}(x,y)+[x^{n-2}y]\mathscr{T}_{P}(x,y)+[x^{n-2}y^{2}]\mathscr{T}_{P}(x,y)$
 and $[y^{n-2}]\mathscr{T}_{P}(1,y)=[x^0y^{n-2}]\mathscr{T}_{P}(x,y)+[xy^{n-2}]\mathscr{T}_{P}(x,y)+[x^{2}y^{n-2}]\mathscr{T}_{P}(x,y)$. The conclusion follows from Lemma \ref{coe-degree-n}, Theorems \ref{coe-degree-n-1'} and \ref{coe-degree-n-2}
 \end{proof}

For hypergraphical cases, Corollary \ref{n-2-coe} (i) is consistent with the result in \cite[Proposition 5.5]{Kalman3}.

\begin{corollary} \cite{Kalman3} \label{complete-example}
Let $\mathcal{H}$ be a hypergraph and  $\mathrm{Bip} \mathcal{H}=(V\cup E,\mathcal{E})$ be its associated bipartite graph. Then $[x^{2}]I_{\mathcal{H}}(x)=\binom{|\mathcal{E}|-|V|-|E|+2}{2}-N$, where $N$ is the number of 4-cycles in $\mathrm{Bip} \mathcal{H}$.
\end{corollary}

\subsection{A connectivity condition for hypergraphs}
\label{sec5.2}

K\'{a}lm\'{a}n \cite{Kalman1} computed the exterior polynomials of hypergraphs induced by complete bipartite graphs.

\begin{observation} \cite{Kalman1} \label{complete-example}
Let $P$ be a polymatroid on $[n]$  with rank function $f$. If $f(I)=f([n])$ for all nonempty subsets $I\subseteq [n]$, then $[y^{i}]X_{\mathcal{H}}(y)=\binom{f([n])+i-1}{i}$ for all nonnegative integers $i\leq n-1$.
\end{observation}

By Corollary \ref{HG-monotonicity} and Observation \ref {complete-example},  the coefficient of the degree $k$ term of the exterior polynomial of any hypergraph $\mathcal{H}=(V,E)$ is at most $\binom{|V|+k-2}{k}$. We now characterize hypergraphs attaining this maximal value, for larger and larger sets of consecutive exponents $k$.
We first show the key lemma in this subsection.

\begin{lemma} \label{connected-exterior}
Let $P\subseteq \mathbb{Z}^{n}_{\geq 0}$ be a polymatroid with  rank function $f$. Then for all $k\leq n-1$, the coefficient $[y^{i}]X_{P}(y)=\binom{f([n])+i-1}{i}$ for all $i\leq k$ if and only if $f([n]\setminus J)=f([n])$ for all $J\subseteq [n]$ with $|J|=k$.
\end{lemma}

\begin{proof}
It is obvious for the case $k=0$ since the constant term of the exterior polynomial for any polymatroid is $1$. We now assume that $k\geq 1$.
We first prove the sufficiency by induction on $n$. If $n=k+1$, then it follows from Observation \ref {complete-example}.

Assume that $n>k+1$. Note that if $P\subseteq \mathbb{Z}^{n}_{\geq 0}$, then $f$ is non-decreasing by \eqref{RF}, that is, $f(A)\leq f(B)$ if $A\subseteq B\subseteq [n]$. This implies that $f([n]\setminus I)=f([n])$ for all subsets $I\subseteq [n]$ with $|I|\leq k$. Hence, $T_{t}=\{0,1,2,\ldots, f(\{t\})\}$ for any $t\in [n]$. For any $j\in T_{t}$, let $f^{t}_{j}$ be the rank function of the polymatroid $\widehat{P}^{t}_{j}$.
\begin{enumerate}
\item [(i)] For any subset $J'\subseteq [n]\setminus \{t\}$ with $|J'|= k-1$, we have that $$f([n]\setminus J')-j=f([n])-j\leq f([n])=f([n]\setminus \{t\}\setminus J').$$ Then
\ $f^{t}_{j}([n]\setminus \{t\}\setminus J')=f([n])-j=f^{t}_{j}([n]\setminus \{t\})$ by Proposition \ref{rank function of Pj}.
\item [(ii)] By Proposition \ref{rank function of Pj}, for any subset $J''\subseteq [n]\setminus \{t\}$ with $|J''|=k$, we have that
$$f^{t}_{f(\{t\})}([n]\setminus \{t\}\setminus J'')=f([n]\setminus J'')-f(\{t\})=f([n])-f(\{t\})=f^{t}_{f(\{t\})}([n]\setminus \{t\}).$$
\end{enumerate}

By the induction hypothesis, for any nonnegative integer $i\leq k$, we have $[y^{i}]X_{\widehat{P}^{t}_{f(\{t\})}}(y)=\binom{f([n])+i-1-f(\{t\})}{i}$ and $[y^{i-1}]X_{\widehat{P}^{t}_{j}}(y)=\binom{f([n])-j+i-1-1}{i-1}$ for any $j\in T_{t}\setminus \{f(\{t\})\}$. Hence, by Corollary \ref{deletion-contraction-IX}, we know that
\begin{eqnarray*}
[y^{i}]X_{P}(y)&=&[y^{i}]X_{\widehat{P}^{t}_{f(\{t\})}}(y)+\sum_{j\in T_{t}\setminus \{f(\{t\})\}}[y^{i-1}]X_{\widehat{P}^{t}_{j}}(y)\\
&=&\binom{f([n])+i-1-f(\{t\})}{i}+\sum_{j\in T_{t}\setminus \{f(\{t\})\}}\binom{f([n])-j+i-2}{i-1}\\
&=&\binom{f([n])+i-1}{i}.
\end{eqnarray*}

For the necessity, by induction on $n$, we prove that $[y^{i}]X_{P}(y)<\binom{f([n])+i-1}{i}$ for some $i\leq k$ if $f([n]\setminus T')<f([n])$ for some subset $T'$ of size $k$.
If $n=k+1$, then this is clear as $|P|<\sum\limits_{i=0}^{k}\binom{f([n])+i-1}{i}$. (This is because if $\mathbf{a}\in \mathbb{Z}_{\geq 0}^{n}$ satisfies $\sum\limits_{i\in [n]\setminus T'}a_{i}=f([n])$ then $\mathbf{a}\notin P$.)

Assume that $n>k+1$. Let $t'$ be an element of $[n]\setminus T'$. Then  $T'\subseteq [n]\setminus\{t'\}$ satisfies $f^{t'}_{f(\{t'\})}([n]\setminus \{t'\}\setminus T')=f([n]\setminus T')-f(\{t'\})<f([n])-f(\{t'\})=f^{t'}_{f(\{t'\})}([n]\setminus \{t'\})$.
By the induction hypothesis, $[y^{i}]X_{\widehat{P}^{t'}_{f(\{t'\})}}<\binom{f([n])+i-1-f(t')}{i}$ for some $i\leq k $. By Theorem \ref{subset-monotonicity} and Observation \ref {complete-example}, for all $j\in T_{t'}\setminus \{f(\{t'\})\}$, the coefficient $[y^{i-1}]X_{\widehat{P}^{t'}_{j}}\leq\binom{f([n])-j+i-2}{i-1}$ as $\widehat{P}^{t'}_{j}\subseteq \mathbb{Z}^{n-1}_{\geq 0}$. By Corollary \ref{deletion-contraction-IX},
$$
[y^{i}]X_{P}(y)=[y^{i}]X_{\widehat{P}^{t'}_{f(\{t\})}}(y)+\sum_{j\in T_{t'}\setminus \{f(\{t'\})\}}[y^{i-1}]X_{\widehat{P}^{t'}_{j}}(y)<\binom{f([n])+i-1}{i},
$$
a contradiction.  This completes the proof.
\end{proof}

\begin{theorem}\label{connected-exterior2}
Let $\mathcal{H}=(V,E)$ be a hypergraph and let $\mathrm{Bip} \mathcal{H}$ be its associated bipartite graph. Then for all $k\leq |E|$, we have $[y^{i}]X_{\mathcal{H}}=\binom{|V|+i-2}{i}$ for all $i\leq k$ if and only if $\mathrm{Bip} \mathcal{H}-E'$ is connected for any $E'\subseteq E$ with $|E'|=k$.
\end{theorem}

\begin{proof}
The statement follows from Lemma \ref{connected-exterior} since if $\mathrm{Bip} \mathcal{H}-E'$ is connected, then $f(E\setminus E')=f(E)=|V|-1$, where $f$ is the rank function of $P_{\mathcal{H}}$.
\end{proof}

\begin{corollary}
  For a graph $G=(V,E)$ and an integer $k\geq 0$, the equality $[y^{i}]T_{G}(1,y)=\binom{|E|-i-1}{|V|-2}$ holds for all $i$ with $|E|-|V|+1-k\leq i\leq |E|-|V|+1$ if and only if $G$ is $(k+1)$-edge connected.
\end{corollary}

\begin{remark}
Let $\mathcal{H}=(V,E)$ be a hypergraph and $\mathrm{Bip} \mathcal{H}$ be its associated bipartite graph. For any $k\geq0$, if there is a subset  $E'\subseteq E$ with $|E'|=k$ such that $\mathrm{Bip} \mathcal{H}-E'$ is connected and there are at least two edges incident with $e$, in $\mathrm{Bip} \mathcal{H}$, for all $e\in E'$, then $[y^{i}]X_{\mathcal{H}}>0$ for all $i\leq k$.

Indeed, given an order of $E$ such that the elements of $E'$ are the last $|E'|$ elements of $E$, there exists a vector $\mathbf{a}=(a_{1},\ldots,a_{|E|-|E'|},0,\ldots,0)$) which is a basis of the hypergraphical polymatroid $P_{\mathcal{H}}$, since $\mathrm{Bip} \mathcal{H}-E'$ is connected. Let $f$ be the rank function of $P_{\mathcal{H}}$. Then for all $e\in E'$ and for all subsets $E''\subseteq E$ with $e=\min(E'')$, we have that $f(E'')>0$ since there are at least two edges incident with $e$. Note that $E''\subseteq E'$ and that $\sum_{e'\in E''}a_{e'}=0<f(E'')$. Thus $e\notin \mathrm{Ext}_{P_{\mathcal{H}}}(\mathbf{a})$ and $\epsilon(\mathbf{a})\leq |E|-k$. By the interpolatory property of the exterior polynomial (see \cite[Theorem 14]{Guan3}), we have that $[y^{i}]X_{\mathcal{H}}>0$ for all $i\leq k$.
\end{remark}

\section{Concluding remarks}
\label{sec6}

The main use of deletion-contraction, for graphs and, to a lesser degree, matroids, is to compute the Tutte polynomial recursively.
Our deletion-contraction formula, Theorem \ref{deletion-contraction-J}, is a sum over parallel slices of a polymatroid. It operates within the category of polymatroids and complements the activity-based Definition \ref{pop} and the state sum given in \cite[Theorem 10.6]{Bernardi} (whose underlying idea is a version of the Crapo interval subdivision) with a third, recursive approach.

However, even if a polymatroid is hypergraphical, we do not see a natural reason for all of its slices, apart from the two extremal ones, to be themselves hypergraphical. Hence, we pose the following question.

\begin{question}
Are all slices of a hypergraphical polymatroid also hypergraphical polymatroids?
\end{question}

A positive answer would be desirable in order to have a recursive method of computation that is fully within the realm of hypergraphs. In case the answer is negative, which appears likely, we relax our question as follows.

\begin{question}
Can we express $\mathscr{T}_{P_{\mathcal{H}}}(x,y)$, for an arbitrary hypergraph $\mathcal{H}$, in terms of  polymatroid Tutte polynomials of some (in a suitable sense) smaller hypergraphs?
\end{question}

\begin{notes}[Acknowledgements]
This paper was written while the first author visited the then Tokyo Institute of Technology, currently Institute of Science Tokyo, in 2022-2023. It is a pleasure to acknowledge the hospitality, as well as the financial support from the China Scholarship Council (No.\ 202206310079), that made the visit possible. XG was also partially supported by the National Natural Science Foundation of China (No.\ 12401462).

XJ was supported by the National Natural Science Foundation of China (No.\ 12171402).

TK was supported by consecutive Japan Society for the Promotion of Science (JSPS) Grants-in-Aid for Scientific Research C (Nos.\ 17K05244 and 23K03108).
\end{notes}


\end{document}